\documentclass[10pt,preprint]{amsart}
\usepackage{amsmath,amsthm,amsfonts,amssymb}
\usepackage{mathtools} 
\usepackage{mathrsfs} 
\usepackage[active]{srcltx} 
\usepackage{amsaddr}

\usepackage[numeric]{amsrefs}

\usepackage{adjustbox}
\usepackage[myheadings]{fullpage}

\theoremstyle{plain} %% This is the default, anyway

\theoremstyle{definition}

\theoremstyle{remark}

\newtheorem{theorem}{Theorem}[section]

\begin{document}

%%% In the title, use a double backslash "\\" to show a linebreak:
%%% Use one of the following two forms:
%%% \title{Text of the title}
%%% or
%%% \title[Short form for the running head]{Text of the title}
\title{ Hyperbolic model for Helmholtz equation with impedance boundary conditions}

%%% If there are multiple authors, they're described one at a time:
%%% First author: 
\author{Ramaz Botchorishvili} 
\address{Faculty of Exact and Natural Sciences, Ivane  Javakhishvili Tbilisi State University, Tbilisi, Georgia} 
\email[Corresponding author]{ramaz.botchorishvili@tsu.ge (Corresponding author)} 
%%% Second author: 
\author{Tamar Janelidze} 
\address{Faculty of Exact and Natural Sciences, Ivane  Javakhishvili Tbilisi State University, Tbilisi, Georgia;
         Institute of Energy and Climate research 8: Troposphere, Research Center J\"ulich, J\"ulich, Germany } 
\email{t.janelidze@fz-juelich.de} 
%%% Third author: \author{} \address{} \curraddr{} \email{} \thanks{}

%%% Your address, or department:
%\address{Department of Mathematics, Eastern Upper Fake University, Nowhereville, SR 99999}

%%% Current address is optional.
% \curraddr{curraddr}

%%% Email address is optional.
%\email{name@fake.address}

%%% Website address is optional.
%\urladdr{http://name.fake.address/}

%%% Thanks acknowledges financial support, e.g., grant money
% \thanks{thanks}

%%% If there's a second author:
% \author{}
% \address{}
% \curraddr{}
% \email{}
% \urladdr{website}
% \thanks{thanks}

%%% AMS subject classification
%53C44    	Geometric evolution equations (mean curvature flow, Ricci flow, etc.)

%%% The date, which can also be entered manually:
\date{\today}

%%% The abstract:
\begin{abstract}
		Solution of Helmholtz equation with impedance boundary condition on finite interval is equivalently reformulated as 
		steady state of initial boundary value problem for first order hyperbolic system of partial differential equations. 
		Particularly interesting property of the proposed hyperbolic model is that steady state is achieved in finite time. 
		For large wavenumber the numerically challenging task for Helmholtz equation is achieving high accuracy with small number of nodal points. We successfully solved this problem by means of using  well balanced scheme approach. 
		Numerical tests demonstrate excellent computational potential of the proposed method: high accuracy is achieved for large wavenumber with small number of nodal points in space and time.
\end{abstract}

% \begin{keyword}	 
% Helmholtz equation \sep first order hyperbolic system approach \sep well balanced property
% \end{keyword}
%%% This puts the title, author information, and date on the first page.
\maketitle

%%%-------------------------------------------------------------------
%%%-------------------------------------------------------------------
%%% Start the body of the paper here:
\section{Introduction}
We consider Helmholtz equation on finite interval with impedance boundary consitions
\begin{equation}\label{eq:Helmholz}
	\frac{d^2 u}{d x^2}+k^2 u=f, ~ 0<x<1,
\end{equation}
\begin{equation}\label{eq:ImpedanceCondition}
u'(0) + iku(0)=g_0, ~~~ u'(1) - iku(1)=g_1, 
\end{equation}
\begin{equation}\label{eq:complexFunctions}
u(x)= u_R(x) + iu_I(x),   f(x)= f_R(x) + if_I(x),  g_0= g_{0R} + ig_{0I},  g_0= g_{1R} + ig_{1I},
\end{equation}
where  $i$ is unit imaginary number, $i^2=-1$, $k, g_{RI}, g_{1I}$ are real numbers, $k>0$, $f_R(x), f_I(x)$ are sufficiently smooth real 
valued functions that ensure the functions $u_R(x), u_I(x)$ are sufficiently smooth. 

Main numerical challenge associated with Helmholtz equation is finding such spatial discretization that ensure that obtained linear 
system can be solved efficiently by iterative methods. If $\Delta x$ is discretization step and wavenumber is small then $k \Delta x = const$ 
is a good choice for determining suitable mesh ~ \cite{ihlenburg1995finite}. If wavenumber is large then for dealing with the so %\cite{IhlenburgBabuska:FE }
called pollution effect $k^{\gamma} \Delta x $ small is needed with  $\gamma >1$ resulting in very fine mesh and  very large system of 
linear algebraic equations ~\cite{ihlenburg1995finite}. Solving this linear system by iterative methods is a difficult task, %\cite{IhlenburgBabuska:FE }
see ~\cite{ernst2012difficult} for detailed exposition on the subject.  %\cite{ErnstGander:WhyDifficult}
 
Here we propose new method that is motivated by first order system approach ~\cite{nishikawa2007first} introduced initially for % \cite{Nishikava:Diffusion}
diffusion equation and by well balanced schemes for hyperbolic conservation laws with source terms pioneered in 
 ~\cite{bermudez1994upwind, greenberg1996well}. The strategy set by first order system approach is to solve an equivalent % \cite{Vazgez1994, Greenberg1996}
first-order hyperbolic system instead of the second-order diffusion equation thus achieving stable computation with time step $O(\Delta x)$ 
that is typical for hyperbolic equations instead of time step $O((\Delta x)^2)$ that is typical for parabolic equations. The srtategy set 
by  the well-balanced property of numerical schemes is  preservation of discrete equilibrium states that results in high accuracy with 
small number of nodal points. In our approach we offer suitable combination of these approaches resulting in highly efficient numerical method. 

The rest of the paper is organized as follows: hyperbolic model is developed in section \ref{sec:hypmodel}; numerical scheme is developed in section \ref{sec:numerical}; 
it is investigated theoretically and numerically in sections \ref{sec:numerical} and \ref{sec:numtests}.

%\vspace{-0.3cm}

\section{Hyperbolic model}
\label{sec:hypmodel}
Following the strategy set by first order system approach ~\cite{nishikawa2007first} the goal is finding  first order hyperbolic system of %\cite{Nishikava:Diffusion}
partial differential equations that at steady states is equivalent to Helmholtz equation (\ref{eq:Helmholz}). 
We consider the following linear hyperbolic system 

%\vspace{-0.2cm}
\begin{equation}\label{eq:HypSystemQ}
\frac{\partial \vec{Q}}{\partial t}+A\frac{\partial\vec{Q}}{\partial x}=B\vec{Q}+\vec{F}, ~t>0,
\end{equation}
where $\vec{Q}=(Q_1(t,x), Q_2(t,x), Q_3(t,x), Q_4(t,x))^T$ is unknown real valued vector function, \\ $\vec{F}=(F_1(x), F_2(x), F_3(x), F_4(x))^T$ 
is real valued vector function, $A,B \in \mathbb{R}^{4\times4}$. 

The system (\ref{eq:HypSystemQ}) is equipped with initial and boundary conditions 

%\vspace{-0.2cm}
\begin{equation}\label{eq:InitConditionQ}
\vec{Q}(0,x)=\vec{Q}_0(x), ~~0<x<1,
\end{equation}

%\vspace{-0.2cm}
\begin{equation}\label{eq:BoundCondition,Q}
B_0\vec{Q}(t,0)=\vec{G}_0, ~~B_1\vec{Q}(t,1)=\vec{G}_1, ~~t \geq 0,
\end{equation}
where $\vec{Q}_0(x)$ is sufficiently smooth, $B_0,B_1 \in \mathbb{R}^{4\times4}$, $\vec{G}_0,\vec{G}_1 \in \mathbb{R}^{4}$. We set

%\vspace{-0.2cm}
\begin{equation}\label{eq:lambda<0,>0}
\lambda_1,\lambda_2 >0, \lambda_3,\lambda_4 <0, 
\end{equation}
\begin{equation}\label{eq:G0,G1}
\vec{G}_0=(g_{0R},g_{0I},0,0)^T, ~~	\vec{G}_1=(0,0,g_{1R},g_{1I})^T,
\end{equation}

%\vspace{-0.2cm}
\begin{equation}\label{eq:B_0,B_1}
B_0 = 
\begin{pmatrix}
k    &  0    &  1   &  0     \\
0    &  k    &  0   &  1     \\
0    &  0    &  0   &  0  \\
0    &  0    &  0   &  0  \\
\end{pmatrix}
,
B_1 = 
\begin{pmatrix}
0    &  0    &  0   &  0  \\
0    &  0    &  0   &  0  \\
0    &  k    &  1   &  0     \\
-k   &  0    &  0   &  1     \\
\end{pmatrix}
,
\end{equation}
\begin{equation}\label{eq:RHS,F}
\vec{F}= \frac{1}{2}
\begin{pmatrix}
[f_I (\lambda_2-\lambda_4)+f_R (\lambda_1-\lambda_3)] \frac{1}{k} \\
[f_I ({\lambda}_2-{\lambda}_4)+f_R ({\lambda}_3-{\lambda}_1)] \frac{1}{k} \\
f_I (\lambda_4 -\lambda_2)+f_R (\lambda_1+\lambda_3) \\
f_I (\lambda_2+\lambda_4)+f_R (\lambda_1-\lambda_3) 
\end{pmatrix}
,
\end{equation}
\begin{equation}\label{eq:LHS,A}
A = \frac{1}{2}
\begin{pmatrix}
\lambda_1+\lambda_4    &  \lambda_2-\lambda_3    & (\lambda_1-\lambda_3)\frac{1}{k}  &  (\lambda_2-\lambda_4)\frac{1}{k} \\
\lambda_4-\lambda_1    &  \lambda_2+\lambda_3    & (\lambda_3-\lambda_1)\frac{1}{k}  &  (\lambda_2-\lambda_4)\frac{1}{k} \\
(\lambda_1-\lambda_4)k &  (\lambda_3-\lambda_2)k & \lambda_1+\lambda_3            &  \lambda_4-\lambda_2  \\
(\lambda_1-\lambda_4)k &  (\lambda_2-\lambda_3)k & \lambda_1-\lambda_3            &  \lambda_4+\lambda_2  \\
\end{pmatrix}
,
\end{equation}
\begin{equation}\label{eq:RHS,B}
B = \frac{1}{2}
\begin{pmatrix}
(\lambda_3-\lambda_1)k    &  (\lambda_4-\lambda_2)k    &  \lambda_1+\lambda_4   &  \lambda_2-\lambda_3     \\
(\lambda_1-\lambda_3)k    &  (\lambda_4-\lambda_2)k    & \lambda_4-\lambda_1    &  \lambda_2+\lambda_3     \\
-(\lambda_1+\lambda_3)k^2 &  (\lambda_2-\lambda_4)k^2  & (\lambda_1-\lambda_4)k &  (\lambda_2-\lambda_3)k  \\
(\lambda_3-\lambda_1)k^2  &  -(\lambda_2+\lambda_4)k^2 & (\lambda_1-\lambda_4)k &  (\lambda_2-\lambda_3)k  \\
\end{pmatrix}
.
\end{equation}
\begin{theorem}[Steady state equivalence]\label{Th:equivalence}
	Helmholtz equation with impedance boundary conditions (\ref{eq:Helmholz}),(\ref{eq:ImpedanceCondition}) is equivalent to initial 
	boundary value problem for hyperbolic system (\ref{eq:HypSystemQ})-(\ref{eq:RHS,B}) at steady states in the following sense: 
\begin{equation}\label{eq:U=Q}
	u_R(x)=Q_1(t,x), ~u_I(x)=Q_2(t,x),~\frac{du_R(x)}{dx}=Q_3(t,x), ~\frac{du_I(x)}{dx}=Q_4(t,x), ~t \geq 0.
\end{equation}
\end{theorem}
\begin{proof} Helmholtz equation (\ref{eq:Helmholz}) equivalently writes as the following first order system of equations:
\begin{equation}\label{eq:Helmholtz,q}
	\frac{d\vec{q}}{dx}=B_q \vec{q}+\vec{F_q},
\end{equation}
\begin{equation}\label{eq:v,q}
v_R=\frac{du_R}{dx}, ~~	v_I=\frac{du_I}{dx}, ~~ \vec{q}=(u_R,u_I,v_R,v_I)^T
\end{equation}
\begin{equation}\label{eq:B_q,F_q}
B_q= \begin{pmatrix}
	0    &  0   & 1 & 0 \\
	0    &  0   & 0 & 1 \\
	-k^2 &  0   & 0 & 0 \\
	0    & -k^2 & 0 & 0\\
	\end{pmatrix}
, ~~~~ \vec{F_q} = 
	\begin{pmatrix}
	0   \\
	0   \\
	f_R \\
	f_I \\
	\end{pmatrix}
	.
\end{equation}

From the system (\ref{eq:Helmholtz,q})-(\ref{eq:B_q,F_q}) one can easily recover Helmholtz equation (\ref{eq:Helmholz}) by means of using first and second equations of (\ref{eq:Helmholtz,q}) in the third and fourth equations. 

Using the above notations boundary condition (\ref{eq:ImpedanceCondition}) equivalently writes:
\begin{equation}\label{eq:BoundCondition,q}
B_0\vec{q}(0)=\vec{G}_0, ~~B_1\vec{q}(1)=\vec{G}_1.
\end{equation}

At steady states ${\partial Q_i(t,x) / \partial t} =0, ~i=1,2,3,4$ and hyperbolic system (\ref{eq:HypSystemQ}) is reduced to the following system of equations: 
\begin{equation}\label{eq:Helmholtz,ODE,Q,1}
A\frac{\partial\vec{Q}}{\partial x}=B\vec{Q}+\vec{F}.
\end{equation}

Inverse of $A$ exist, since $det(A)=\lambda_1 \lambda_2 \lambda_3 \lambda_4$ and $det(A) \neq 0$ because of (\ref{eq:lambda<0,>0}). 
Multiplying (\ref{eq:Helmholtz,ODE,Q,1}) on $A^{-1}$ from the left yields:
\begin{equation}\label{eq:Helmholtz,ODE,Q,2}
\frac{\partial \vec{Q}}{\partial x}=B_q \vec{Q}+\vec{F_q}.
\end{equation}

At steady states $\vec{Q}$ does not depend on variable $t$  and it depends on variable $x$ only. Therefore (\ref{eq:Helmholtz,q}) and (\ref{eq:Helmholtz,ODE,Q,2}) 
are in fact the same equations when ${\partial Q_i(t,x) / \partial t} =0, ~i=1,2,3,4$ and (\ref{eq:U=Q}) is also valid. For the same reasons boundary conditions 
(\ref{eq:BoundCondition,q}) and (\ref{eq:BoundCondition,Q}) are equivalent when $\vec{Q}$ does not depend on variable $t$.  

Thus we have shown that Helmholtz equation with impedance boundary conditions  (\ref{eq:Helmholz}),(\ref{eq:ImpedanceCondition})  and hyperbolic system of partial differential equations with initial and boundary conditions (\ref{eq:HypSystemQ})-(\ref{eq:RHS,B}) are equivalent to the same problem when ${\partial Q_i(t,x) / \partial t} =0, ~i=1,2,3,4$ that concludes the proof. 
\end{proof}

%\vspace{-0.2cm}
\begin{theorem}[Stationary solution of hyperbolic model]\label{Th:Stationary}
	Smooth solution of the problem (\ref{eq:HypSystemQ})-(\ref{eq:RHS,B}) obtained with arbitrary initial value $\vec{Q}_0(x)$ reaches steady state in 
	finite time if $\lambda_1=\lambda_2$ and $\lambda_3=\lambda_4$.
\end{theorem}
\begin{proof} Hyperbolic model (\ref{eq:HypSystemQ})-(\ref{eq:RHS,B}) equivalently writes:
\begin{equation}\label{eq:HypSystem,r}
     \frac{\partial \vec{r}}{\partial t}+\Lambda\frac{\partial\vec{r}}{\partial x}=B_r\vec{r}+\vec{F}_r, ~t>0,
\end{equation}
\begin{equation}\label{eq:InitCondition,r}
     \vec{r}(0,x)=\vec{r}_0(x), ~~0<x<1,
\end{equation}
\begin{equation}\label{eq:BoundCondition,r}
  r_1(t,0)=g_{0R}, ~r_2(t,0)=g_{0I}, ~r_3(t,1)=g_{1R}, ~r_4(t,1)=g_{1I},  ~t \geq 0,
\end{equation}
where $\vec{r}$ is Riemann invariant,  
\begin{equation}\label{eq:r,Lambda}
	\vec{r}=L\vec{Q}, ~~L=B_0+B_1, ~~\Lambda=diag\{\lambda_1, \lambda_2, \lambda_3,\lambda_4 \},
\end{equation}
\begin{equation}\label{eq:B_r,F_r}
B_q= k \begin{pmatrix}
0            &  -\lambda_1   & \lambda_1  & \lambda_1 \\
\lambda_2    &           0   & -\lambda_2 & \lambda_2 \\
0            &           0   &          0 & \lambda_3 \\
0            &           0   & -\lambda_4 &         0\\
\end{pmatrix}
, ~~~~ \vec{F_r} = 
\begin{pmatrix}
\lambda_1 f_R \\
\lambda_2 f_I \\
\lambda_3 f_R \\
\lambda_4 f_I \\
\end{pmatrix}
.
\end{equation}

Riemann invariant $\vec{r}$ is smooth function, $\Lambda, B_r, \vec{F}_r$ do not depend on the variable $t$. Therefore for $\vec{r}~'\equiv {\partial r / \partial t}$ 
from (\ref{eq:HypSystem,r})-(\ref{eq:B_r,F_r}) we obtain the following problem: 
\begin{equation}\label{eq:HypSystem,r'}
\frac{\partial \vec{r}~'}{\partial t}+\Lambda\frac{\partial\vec{r}~'}{\partial x}=B_r\vec{r}~',
\end{equation}
\begin{equation}\label{eq:InitCondition,r'}
\vec{r}~'(0,x)=\vec{r}~'_0(x), ~~0<x<1,
\end{equation}
\begin{equation}\label{eq:BoundCondition,r'}
r'_1(t,0)=0, ~r'_2(t,0)=0, ~r'_3(t,1)=0, ~r'_4(t,1)=0,  ~t \geq 0.
\end{equation}  	

 Third and fourth equations of (\ref{eq:HypSystem,r'}) are multiplied by $r'_3$ and $r'_4$ respectively. Summing obtained equations yields: 
\begin{equation}\label{eq:r3'^2+r4'^2}
	\frac{\partial [(r'_3)^2+(r'_4)^2] }{\partial t}+\lambda_3\frac{\partial[(r'_3)^2+(r'_4)^2]}{\partial x}=0.
\end{equation}

From (\ref{eq:r3'^2+r4'^2}) and (\ref{eq:BoundCondition,r'}) we obtain $(r'_3)^2+(r'_4)^2=0$ when $t>1/|\lambda_3|$. Analogically for $(r'_1)^2+(r'_2)^2$ we obtain the following equation:
\begin{equation}\label{eq:r1'^2+r2'^2}
	\frac{\partial [(r'_1)^2+(r'_2)^2] }{\partial t}+\lambda_1\frac{\partial[(r'_1)^2+(r'_2)^2]}{\partial x}=
	2\lambda_1(r'_3+r'_4)r'_1-2\lambda_2(r'_3-r'_4)r'_2.
\end{equation}

From (\ref{eq:r1'^2+r2'^2}), (\ref{eq:BoundCondition,r'}) and from $r'_3=r'_4=0$ when $t>1/|\lambda_3|$ we obtain $(r'_1)^2+(r'_2)^2=0$ when $t>1/|\lambda_3|+1/\lambda_1$. 
Thus ${{\partial r_i} / {\partial t}} =0, i=1,2,3,4$, when $t>1/|\lambda_3|+1/\lambda_1$, i.e. steady state is reached in finite time.
\end{proof}

%\vspace{-0.2cm}

\section{Numerical scheme}
\label{sec:numerical}
Following ~\cite{botchorishvili2003equilibrium} we consider numerical scheme % \cite{bpv2003}
%\vspace{-0.3cm}
\begin{equation}\label{ns:sc,r}
	\frac{\vec{r}_j^{~n+1}-\vec{r}_j^{~n}}{\Delta t}+\Lambda^+\frac{\vec{r}_j^{~n}-\vec{r}_{j-1,+}^{~n}}{\Delta x}
	+\Lambda^-\frac{\vec{r}_{j+1,-}^{~n}-\vec{r}_{j}^{~n}}{\Delta x}=0,
\end{equation}
%\vspace{-0.9cm}
\begin{equation}\label{ns:InitialCondition,r}
\vec{r}_j^{~0}=\vec{r}(0,x_j), ~x_j=j\Delta x, ~j=0,1,..,N_x, ~\Delta x=1/N_x,
\end{equation}
\begin{equation}\label{ns:BoundaryCondition,r}
r_{1,0}^{n}=g_{0R}, ~ r_{2,0}^{n}=g_{0I}, ~r_{3,N_x}^{n}=g_{1R}, ~r_{4,N_x}^{n}=g_{1I}, 
\end{equation}
%\vspace{-0.4cm}
\begin{equation}\label{ns:r+}
	\vec{r}_{j-1,+}^{~n}=e^{B_r\Delta x}\vec{r}_{j-1}^{~n}+
	e^{B_r\Delta x} \int_{x_{j-1}}^{x_j} e^{-B_r(\xi-x_{j-1})}\vec{F}_r(\xi)d\xi,
\end{equation}
%\vspace{-0.4cm}
\begin{equation}\label{ns:r-}
\vec{r}_{j+1,-}^{~n}=e^{-B_r\Delta x}\vec{r}_{j+1}^{~n}+
e^{-B_r\Delta x} \int_{x_{j+1}}^{x_j} e^{-B_r(\xi-x_{j+1})}\vec{F}_r(\xi)d\xi,
\end{equation}
%\vspace{-0.3cm}
\begin{equation}\label{ns:Lambda+,Lambda-}
\Lambda^+=diag \{\lambda_1,\lambda_2,0,0\},~\Lambda^-=diag \{0,0,\lambda_3,\lambda_4 \}, 
\end{equation}
\begin{equation*}
t_n=n\Delta t, ~n=0,1,..,N_t, ~\Delta t=T/N_t, ~T>0, N_t,N_x \in \mathbb{N}.
\end{equation*}
%\vspace{-0.8cm}
\begin{theorem}[Properties of the scheme]
	Numerical scheme (\ref{ns:sc,r})-(\ref{ns:Lambda+,Lambda-}) maintains discrete steady states of (\ref{eq:HypSystem,r})-(\ref{eq:BoundCondition,r}), 
	the scheme is stable under CFL condition $max\{\lambda_1,\lambda_2,|\lambda_3|,|\lambda_4|\}\Delta t/\Delta x <1$ and it is consistent with 
	(\ref{eq:HypSystem,r})-(\ref{eq:BoundCondition,r}) in the sense of local truncation error.
\end{theorem}
%\vspace{-0.5cm}
\begin{proof}
	 Equations (\ref{eq:Helmholtz,q}),(\ref{eq:Helmholtz,ODE,Q,2}) are valid at steady states of  (\ref{eq:HypSystem,r})-(\ref{eq:BoundCondition,r}). 
	 Having discrete steady states given at $t_n$ means that $L\vec{q}(x)$ is projected on mesh, i.e. $\vec{r}_j^{~n}=L\vec{q}(x_j)$. At $x_j$ solution 
	 of (\ref{eq:Helmholtz,q}) with initial condition $\vec{q}(x_{j-1})=L^{-1}\vec{r}_{j-1}^{~n}$ coincides with  $L^{-1}\vec{r}_{j-1,+}^{~n}$ 
	 when $\vec{r}_{j-1,+}^{~n}$ is defined by (\ref{ns:r+}). Therefore $\vec{r}_{j}^{~n}-\vec{r}_{j-1,+}^{~n}=0$ at discrete steady states. 
	 Analogically from (\ref{eq:Helmholtz,q}) and (\ref{ns:r-}) we obtain $\vec{r}_{j+1,-}^{~n}-\vec{r}_{j}^{~n}=0$ is valid at discrete steady 
	 states of (\ref{eq:HypSystem,r})-(\ref{eq:BoundCondition,r}) and thus the scheme  (\ref{ns:sc,r})-(\ref{ns:Lambda+,Lambda-}) is reduced 
	 to $\vec{r}_{j}^{~n+1}-\vec{r}_{j}^{~n}=0$,  i.e. discrete steady states are maintained.
	
	Using Taylor's series expansion in (\ref{ns:r+}),(\ref{ns:r-}) yields the following equivalent formulas:
%\vspace{-0.4cm}
\begin{equation}\label{ns:r+,dx}
	\vec{r}_{j-1,+}^{~n}=\vec{r}_{j-1}^{~n}+(B_r\vec{r}_{j-1}^{~n} + \vec{F}_r(x_j))\Delta x + 0((\Delta x)^2),
\end{equation}
\begin{equation}\label{ns:r-,dx}
	\vec{r}_{j+1,-}^{~n}=\vec{r}_{j+1}^{~n}-(B_r\vec{r}_{j+1}^{~n} + \vec{F}_r(x_j))\Delta x + 0((\Delta x)^2).
\end{equation}
%\vspace{-0.1cm}

	On account of (\ref{ns:r+,dx}) and (\ref{ns:r-,dx}) the scheme (\ref{ns:sc,r}) equivalently writes
%\vspace{-0.1cm}
\begin{equation}\label{ns:sc,r+rhs}
	\frac{\vec{r}_j^{~n+1}-\vec{r}_j^{~n}}{\Delta t}+\Lambda^+\frac{\vec{r}_j^{~n}-\vec{r}_{j-1}^{~n}}{\Delta x}
	+\Lambda^-\frac{\vec{r}_{j+1}^{~n}-\vec{r}_{j}^{~n}}{\Delta x}=
	\Lambda^+B_r\vec{r}_{j-1}^{~n}+\Lambda^-B_r\vec{r}_{j+1}^{~n}+\vec{F}_{r,j} +O(\Delta x).
\end{equation}
%\vspace{-0.1cm}

From (\ref{ns:sc,r+rhs}) obtaining stability and consistency is straightforward when using standard techniques for studying numerical schemes for first order hyperbolic partial differential equations.  	
\end{proof}

%\vspace{-0.3cm}

\section{Numerical tests}\label{sec:numtests}
%\vspace{-0.2cm}
In the first numerical test considered here $\Delta x$ is kept constant and wavenumber $k$ gradually increases from $10$ to $10^{5}$. 
In the second numerical test $k\Delta x$  is kept constant and $\Delta x$ decreases from $10^{-1}$ to $10^{-5}$. We define exact solution of the 
problem (\ref{eq:Helmholz}),(\ref{eq:ImpedanceCondition}) by $u(x)=sin(kx)+i2cos(kx)$ that implies $f(x)=0$ in (\ref{eq:Helmholz}). 
Therefore $\vec{F}_{r}=0$ and integral vanishes in (\ref{ns:r+}), (\ref{ns:r-}). Matrix exponentials present in these formulas are precomputed 
for $k\Delta x =1$. Then for $k\Delta x =m, ~m \in \mathbb{N}$, computation of matrix exponential is reduced to the multiplication of precomputed matrices. 
In hyperbolic model we set $\lambda_1=\lambda_2=1$ and $\lambda_3=\lambda_4=-1$ that according to Theorem ~\ref{Th:Stationary} 
means that steady state is reached at $T=2$. This effect is also observed numerically in our tests. 
In the tables ~\ref{tab:error,q,k} and ~\ref{tab:error,q,dx} the norms $\|\vec{e}\|_{1,2}=max \{ \|e_1\|,\|e_2\| \} $, $\|\vec{e}\|_{3,4}=max \{ \|e_3\|,\|e_4\| \} $ 
are used for measuring the errors in numerical approximation of the function $u$ and in its derivatives respectively. Numerical results show excellent 
computational potential of the developed numerical scheme. In particular the table ~\ref{tab:error,q,k} shows that for wavenumber $k=10^5$  relative 
error  $\approx 10^{-3}$ can be obtained with just $11$ nodal points in space and $20$ time steps. Mesh refinement, see table \ref{tab:error,q,dx}, 
decreases the error further granting high accuracy when $k\Delta x =1$.

%\section*{References}
\bibliography{hyperbolic_model_for_helmholtz_equation_with_impedance_boundary_conditions}

%\vspace{-0.2cm}
\begin{table}[h!]
	\centering
	\caption{ Error $\vec{e}=\vec{q}_{exact}-\vec{q}_{numeric}$, $\Delta x=0.1$, $N_x=10$, $N_t=20$.}
	\begin{adjustbox}{width=1\linewidth}
		\begin{tabular}{c c c c c c}
			\hline	
			 k       & Error,$\vec{e}$     &  $l_{2,rel}$   & $l_{\infty,rel}$ & $l_{2,abs}$ & $l_{\infty,abs}$ \\ 
			\hline
			 10      & $\|\vec{e}\|_{1,2}$  & 3.3035777e-07 & 3.817435e-07  & 1.7788994e-06 & 7.63487e-07  \\
			         & $\|\vec{e}\|_{3,4}$  & 3.4928838e-07 & 4.0443439e-07 & 1.7811752e-05 & 8.0026099e-06 \\
			\hline	
			 $10^2$  & $\|\vec{e}\|_{1,2}$  & 3.1886394e-06 & 3.602908e-06  & 1.6753317e-05  & 7.205816e-06 \\
			         & $\|\vec{e}\|_{3,4}$  & 3.411358e-06  & 3.9042842e-06 & 1.7855065e-03  & 7.7608476e-04 \\
			\hline				
			 $10^3$  & $\|\vec{e}\|_{1,2}$  & 3.9453715e-05 & 4.2069045e-05 & 2.0419553e-04 & 8.413809e-05 \\
			         & $\|\vec{e}\|_{3,4}$  & 3.427245e-05  & 3.8295731e-05 & 1.8204272e+01 & 7.657276e-02   \\
                        \hline
                         $10^4$  & $\|\vec{e}\|_{1,2}$  & 3.2833097e-04 & 3.7296399e-04 & 1.7728889e-03 & 7.4592799e-04 \\
			         & $\|\vec{e}\|_{3,4}$  & 3.5056249e-04 & 4.0213891e-04 & 1.7821279e+01 & 8.0254302 \\
                        \hline
                         $10^5$  & $\|\vec{e}\|_{1,2}$  & 2.8128045e-03 & 3.1224907e-03 & 1.5103248e-02 & 6.2449814e-03  \\
			         & $\|\vec{e}\|_{3,4}$  & 3.434853e-03  & 3.441699e-03  & 1.757117e+03  & 6.882298e+02 \\
                        \hline
		\end{tabular}
	\end{adjustbox}	
	\label{tab:error,q,k}		
\end{table}
%\vspace{-1cm}
\begin{table}[h!]
	\centering
	\caption{ Error $\vec{e}=\vec{q}_{exact}-\vec{q}_{numeric}$, $k\Delta x=1$, $T=2$.}
	\begin{adjustbox}{width=1\linewidth}
		\begin{tabular}{c c c c c c}
			\hline	
			$\Delta x$ & Error,$\vec{e}$ &  $l_{2,rel}$   & $l_{\infty,rel}$ & $l_{2,abs}$ & $l_{\infty,abs}$ \\ 
			\hline
			$10^{-1}$  & $\|\vec{e}\|_{1,2}$  & 3.9453715e-05 & 4.2069045e-05 & 2.0419553e-04 & 8.413809e-05 \\
			           & $\|\vec{e}\|_{3,4}$  & 3.427245e-05  & 3.8295731e-05 & 1.8204272e-01 & 7.657276e-02 \\
			\hline	
			$10^{-2}$  & $\|\vec{e}\|_{1,2}$  & 3.3700138e-06 & 4.3448308e-06 & 5.3714493e-05 & 8.6896615e-06 \\
			           & $\|\vec{e}\|_{3,4}$  & 3.4083229e-06 & 4.2919226e-06 & 5.3992586e-02 & 8.5817495e-03 \\
			\hline
			$10^{-3}$  & $\|\vec{e}\|_{1,2}$  & 3.3997232e-07 & 4.3876751e-07 & 1.7010247e-05 & 8.7753501e-07 \\
			           & $\|\vec{e}\|_{3,4}$  & 3.4011019e-07 & 4.4119365e-07 & 1.7010874e-02 & 8.8237889e-04 \\
                        \hline
                        $10^{-4}$  & $\|\vec{e}\|_{1,2}$  & 3.3956751e-07 & 4.4190141e-07 & 5.3694783e-05 & 8.8380282e-07 \\
			           & $\|\vec{e}\|_{3,4}$  & 3.3959106e-07 & 4.4202167e-07 & 5.3694984e-01 & 8.8403766e-03 \\
                        \hline
                        $10^{-5}$  & $\|\vec{e}\|_{1,2}$  & 3.3491766e-07 & 4.4119696e-07 & 1.6746016e-04 & 8.8239392e-07 \\
			           & $\|\vec{e}\|_{3,4}$  & 3.3491975e-07 & 4.4119442e-07 & 1.6746022e+01 & 8.8238885e-02 \\
                        \hline
		\end{tabular}
	\end{adjustbox}	
	\label{tab:error,q,dx}	
\end{table}
%\vspace{-0.7cm}

% \section*{References}
%\bibliography{mybibfile_HELMHOLZ}

\end{document}